\documentclass[11pt]{amsart}
\usepackage{geometry}                
\usepackage{graphicx}
\usepackage{amssymb}
\usepackage{epstopdf}
\usepackage[colorlinks]{hyperref}
\newtheorem{theorem}{Theorem}[section]
\newtheorem{corollary}[theorem]{Corollary}
\newtheorem{lemma}[theorem]{Lemma}
\newtheorem{proposition}[theorem]{Proposition}
\theoremstyle{definition}

\newtheorem{remark}[theorem]{Remark}
\newtheorem{example}[theorem]{Example}
\numberwithin{equation}{section}
\usepackage[colorlinks]{hyperref}
\usepackage[lite]{amsrefs}
\usepackage{mathpazo}

\title{Region crossing change on surfaces}
\author{Jiawei Cheng}
\address{School of Mathematical Sciences, Beijing Normal University, Beijing 100875, China}
\email{201711130214@mail.bnu.edu.cn}
\author{Zhiyun Cheng}
\address{School of Mathematical Sciences, Beijing Normal University, Beijing 100875, China}
\email{czy@bnu.edu.cn}
\author{Jinwen Xu}
\address{School of Mathematical Sciences, Beijing Normal University, Beijing 100875, China}
\email{201711130184@mail.bnu.edu.cn}
\author{Jieyao Zheng}
\address{School of Mathematical Sciences, Beijing Normal University, Beijing 100875, China}
\email{201711130196@mail.bnu.edu.cn}
\begin{document}

\begin{abstract}
Region crossing change is a local operation on link diagrams. The behavior of region crossing change on $S^2$ is well understood. In this paper, we study the behavior of (modified) region crossing change on higher genus surfaces.
\end{abstract}

\maketitle

\section{Introduction}\label{section1}
In knot theory, an unknotting operation usually means a special local operation in a tangle such that any knot can be transformed to the unknot by a finite sequence of such local operations. One example of unknotting operation is crossing change, due to the fact that ascending knot diagrams represent the unknot. Crossing change is a very important local operation in knot theory. Investigating the behaviors of knot invariants under one crossing change plays an important role in the study of unknotting number. See \cite{Sch1998} for a nice expository survey about crossing change. Besides of crossing change, there are several other well studied unknotting operations in knot theory. For examples, $\sharp$-operation \cite{Mur1985}, $\triangle$-operation \cite{Mur1989} and $n$-gon move \cite{Aid1992} illustrated in Figure \ref{figure1} are all proved to be unknotting operations.
\begin{figure}[h]
\centering
\includegraphics{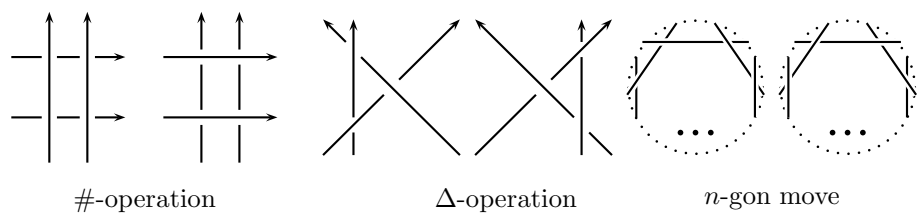}\\
\caption{Some unknotting operations}\label{figure1}
\end{figure}

In this paper we are concerned with a new local operation, say the \emph{region crossing change}, which was first introduced by Ayaka Shimizu in \cite{Shi2014}. Let $K$ be a knot diagram, and $R$ a region of $R^2\setminus K$ (for simplicity, we also call $R$ a region of $K$), then applying region crossing change on $R$ yields a new knot diagram, which is obtained from $K$ by reversing all the crossing points incident to $R$. Figure \ref{figure2} indicates how to convert the trefoil knot into the unknot by applying region crossing change on region $R$. In fact, not only this diagram of trefoil knot can be transformed into a diagram representing the unknot. Let $K$ be a knot diagram and $S$ a set of some crossing points in $K$. Following \cite{Che2013}, we say $S$ is \emph{region crossing change admissible} if one can obtain a new knot diagram $K'$ from $K$ by a finite sequence of region crossing changes, here $K'$ is obtained from $K$ by switching all the crossing points in $S$. As the main result of \cite{Shi2014}, Ayaka Shimizu actually proved the following result.

\begin{figure}
\centering
\includegraphics{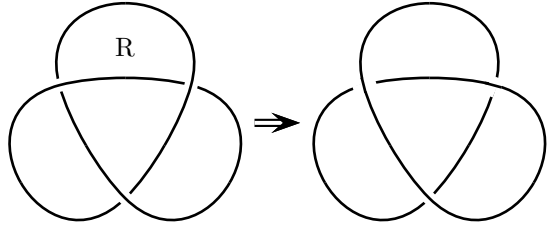}\\
\caption{Region crossing change on $R$}\label{figure2}
\end{figure}

\begin{theorem}[\cite{Shi2014}]\label{theorem1}
Each crossing point in a knot diagram is region crossing change admissible.
\end{theorem}

It follows immediately that region crossing change is an unknotting operation for knot diagrams. In other words, for any knot diagram one can always find some regions of it, such that applying region crossing change on these regions yields a diagram representing the unknot. Note that Reidemeister moves are prohibited during this process, otherwise one can simply use some $\sharp$-operations or $n$-gon moves to untie any given knot, since both of them are special cases of region crossing change. However, it is not difficult to observe that in general region crossing change can not convert every link into a trivial link (consider Hopf link). A natural question is when region crossing change is an unknotting operation for link diagrams. This question was answered by the second author in \cite{Che2013}.

\begin{theorem}[\cite{Che2013}]\label{theorem2}
Let $L=K_1\cup\cdots\cup K_n$ be a link diagram, then region crossing change is an unknotting operation on $L$ if and only if for any $1\leq i\leq n$ we have $\sum\limits_{j\neq i}lk(K_i, K_j)=0$ (mod 2).
\end{theorem}

As a corollary, we find that whether region crossing change is an unknotting operation does not depend on the choice of the diagram. It is an intrinsic property of the link itself. Furthermore, the effect of region crossing change on link diagrams is well understood. For a given link diagram, by ignoring all the crossing information one obtains a 4-valent graph on the plane. We call this graph the \emph{medial graph} or simply the \emph{projection} of the given link diagram. Notice that if two link diagrams can be connected by some region crossing changes then they must have the same projection, since Reidemeister moves are prohibited. Now we assume $L=K_1\cup\cdots\cup K_n$ and $L'=K_1'\cup\cdots\cup K_n'$ are two link diagrams on the plane with the same projection. Then $L$ and $L'$ are related by a finite sequence of region crossing changes if and only if
\begin{center}
$\sum\limits_{j\neq i}lk(K_i, K_j)=\sum\limits_{j\neq i}lk(K_i', K_j')$ (mod 2)
\end{center}
for any $1\leq i\leq n$.

A variety of (modified) region crossing changes can be found in \cite{Aha2012,Hay2015} and \cite{Ino2016}.

Since the behavior of region crossing changes on the plane (or $S^2$) is well understood, in this paper we will move from $S^2$ to closed orientable surfaces with higher genera. We will count the number of equivalence classes of link diagrams with the same projection on surfaces under region crossing changes (with a little modification). Recently, Dasbach and Russell considered an analogous question in \cite{Das2018}. The differences between our results and that in \cite{Das2018} will be discussed in Section \ref{section3} and Section \ref{section4} in detail.

The rest of this paper is organized as follows. Section \ref{section2} introduces the incidence matrix of a link diagram. It turns out that in order to study region crossing change, it suffices to investigate the $\mathbb{Z}_2$-rank of the incidence matrix. In Section \ref{section3} we turn to region crossing changes on surfaces. First we modify the definition of region crossing change a little bit. Then for a given link diagram (or a projection) $L$ we define a graph $G_L$ based on the behavior of region crossing changes. Several properties of this graph will be discussed. Section \ref{section4} is devoted to prove the main result, Theorem \ref{theorem4}. The last section, as its title suggests, contains some results about the  original region crossing changes on surfaces and also some problems we encounter in this case.

\section{Region crossing change and incidence matrix}\label{section2}
For a given graph $G$, let us use $V_G$ and $E_G$ to denote the vertex set and edge set respectively. Recall that the incidence matrix $M_G$ can be defined in the following manner
\begin{center}
$M_G=(m_{x, y}), \quad x\in V_G$ and $y\in E_G$
\end{center}
where
\begin{center}
$m_{x, y}=
\begin{cases}
1& \text{if $y$ is incident to $x$;}\\
0& \text{otherwise.}
\end{cases}$
\end{center}

Now we want to use the incidence matrix of a graph to define an incidence matrix for a link diagram. Let $L$ denote a link diagram. Since the projection of $L$ is a 4-valent planar graph, it follows that it is possible to color all the regions of $L$ in checkerboard fashion. In other words, each region will be colored white or black, such that any two regions share the same arc on the boundary have different colors. Without loss of generality, we always assume that the unbounded region has the white color. Now we assign a vertex to each black region, and if two black regions are connected by several crossing points then we add the same number of edges between the corresponding two vertices. In this way we obtain a planar graph $T_L$, called the \emph{Tait graph} associated to $L$.

\begin{remark}
One can also assign a sign to each edge of $T_L$ to obtain the \emph{signed Tait graph}, from which $L$ can be recovered. But we do not need it here.
\end{remark}

Denote the dual graph of $T_L$ by $T_L'$. Notice that there is a 1-1 correspondence between $E_{T_L}$ and $E_{T_L'}$. By using this correspondence we can put the two incidence matrices together to obtain a new matrix $M_L$ as follows
\begin{center}
$M_L=\begin{pmatrix}
M_{T_L}\\
M_{T_L'}
\end{pmatrix}_{r\times c},$
\end{center}
here $r$ and $c$ denote the number of regions and crossings in $L$ respectively. We name $M_L$ the \emph{incidence matrix} of $L$ and remark that $M_L$ actually is not well-defined. It depends on the order of regions and the order of crossing points. However the rank of it is well defined.

\begin{example}\label{example1}
Consider the figure-8 knot diagram $K$ illustrated in Figure \ref{figure3}, and now we have
\begin{center}
$M_K=\begin{pmatrix}
1&1&0&0&0\\
1&0&1&1&0\\
0&1&1&1&1\\
0&0&0&0&1\\
1&1&1&0&0\\
0&0&1&1&0\\
1&1&0&1&1
\end{pmatrix}.$
\end{center}
\end{example}

\begin{figure}
\centering
\includegraphics{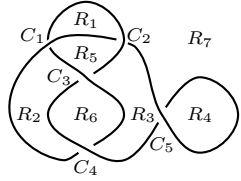}\\
\caption{A figure-8 knot diagram}\label{figure3}
\end{figure}

From Example \ref{example1} it is easy to observe that the incidence matrix is closely related to the operation region crossing change. Let us make it more clear. As before, assume $L$ is a link diagram on the plane (or $S^2$). Let us use $\{R_1, \cdots, R_r\}$ and $\{C_1, \cdots, C_c\}$ to denote the set of regions and the set of crossing points respectively. Now the incidence matrix $M_L$ is a $r\times c$ matrix and each entry belongs to $\mathbb{Z}_2(=\mathbb{Z}/2\mathbb{Z})$, the field consists of two elements 0 and 1. Note that $L$ may be splittable, hence $r\neq c+2$ in general. According to the definition of $M_L$, there exists a one-to-one correspondence between the regions and the row vectors of $M_L$. Moreover, the positions of 1's of a row vector tells us which crossing points will be switched if we apply region crossing change on the corresponding region. In order to understand the effect of region crossing change on some regions, one only needs to read the positions of 1's in the sum of the corresponding row vectors. Throughout the entire paper when we perform row or column operations we always work over $\mathbb{Z}_2$ coefficients unless otherwise specified. Let us abuse our notation, letting $R_i$ refer both to a region of $L$ and to the corresponding row vector of $M_L$. Now Theorem \ref{theorem1} can be restated as, for a knot diagram $K$, each row vector $(0, \cdots, 0, 1, 0, \cdots, 0)$ can be written as the sum of some row vectors of $M_K$. In other words, rank$_{\mathbb{Z}_2}(M_K)=c$. Recall that for a knot diagram we always have $r=c+2$.

\begin{remark}
With a given unknotting operation, it is natural to consider the associated unknotting number. If one defines the unknotting number with respect to region crossing change to be the minimal number of region crossing changes needed to convert a knot $K$ to the unknot among all the knot diagrams of $K$, then any nontrivial knot has unknotting number one \cite{Aid1992}.
\end{remark}

By ignoring the over/undercrossing information of $L$, we obtain the projection of $L$. Since each crossing point can be assigned an overcrossing or an undercrossing, there are totally $2^c$ different link diagrams corresponding to this projection. Assume each component has a fixed orientation, we assign 0 (1) to a crossing point if it is positive (negative). Now each of these $2^c$ link diagrams corresponds to a row vector in $\mathbb{Z}_2^c$. We say two link diagrams are \emph{(region crossing change) equivalent} if they are related by a finite sequence of region crossing changes. Equivalently speaking, two link diagrams with the same projection are equivalent if the difference of the corresponding row vectors can be represented as the sum of some row vectors in $M_L$. The result below follows immediately.

\begin{proposition}\label{proposition1}
Let $L$ be a link diagram with $c$ crossing points, among all the link diagrams with the same projection as $L$, the number of equivalence classes under region crossing changes equals $2^{c-\text{rank}_{\mathbb{Z}_2}(M_L)}$. In particular, every crossing point is region crossing change admissible if and only if $\text{rank}_{\mathbb{Z}_2}(M_L)=c$.
\end{proposition}

Consequently, if one wants to count the number of equivalence classes it suffices to calculate rank$_{\mathbb{Z}_2}(M_L)$. For link diagrams on the plane (or $S^2$), we have the following result, which can be regarded as a generalization of Theorem 1.4 in \cite{Che2012}.

\begin{proposition}\label{proposition2}
Let $L=K_1\cup\cdots\cup K_n$ be a link diagram on the plane, and $r$ the number of regions, then $\text{rank}_{\mathbb{Z}_2}(M_L)=r-n-1$.
\end{proposition}
\begin{proof}
In \cite{Che2012}, it was proved that if the projection of $L$ is connected, then we have $\text{rank}_{\mathbb{Z}_2}(M_L)=c-n+1$. Now let us assume that the projection of $L$ has $s$ components, i.e. $L=L_1\cup\cdots\cup L_s$, where the projection of each $L_i$ $(1\leq i\leq s)$ is a connected component. Assume the sub-link $L_i$ has $c_i$ crossing points. It follows that $r=\sum\limits_{i=1}^s(c_i+2)-(s-1)=\sum\limits_{i=1}^sc_i+s+1=c+s+1$. By suitably choosing the order of regions and the order of crossings, the incidence matrix $M_L$ has the form
\begin{center}
$M_L=\begin{pmatrix}
M_{L_1}'&&&&\\
&M_{L_2}'&&&\\
&&\ddots&&\\
&&&M_{L_{s-1}}'&\\
&&&&M_{L_{s}}'\\
M_{1\times c_1}&M_{1\times c_2}&\cdots&M_{1\times c_{s-1}}&M_{1\times c_s}
\end{pmatrix},$
\end{center}
where
\begin{center}
$M_{L_1}=\begin{pmatrix}
M_{L_1}'\\
M_{1\times c_1}
\end{pmatrix}, \cdots, M_{L_s}=\begin{pmatrix}
M_{L_s}'\\
M_{1\times c_s}
\end{pmatrix}$.
\end{center}
Here the last row vector corresponds to the unbounded region. Since each column vector of $M_L$ contains at least three 1's, it follows that the bases of the column spaces of $\{M_{L_i}\}$ $(1\leq i\leq s)$ form a basis for the column space of $M_L$. Recall that $\text{rank}_{\mathbb{Z}_2}(M_{L_i})=c_i-n_i+1$, where $n_i$ denotes the number of knot components of $L_i$. As a consequence, we have
\begin{center}
$\text{rank}_{\mathbb{Z}_2}(M_L)=\sum\limits_{i=1}^s\text{rank}_{\mathbb{Z}_2}(M_{L_i})=\sum\limits_{i=1}^s(c_i-n_i+1)=c-n+s=r-s-1-n+s=r-n-1.$
\end{center}
\end{proof}

\begin{remark}
Instead of the column space, one can also prove Proposition \ref{proposition2} by considering the row space. Actually, for each $1\leq i\leq s$ the row vector $M_{1\times c_i}$ can be written as a linear combination of some row vectors of $M_{L_i}'$. Let us take a moment to explain the reason. Recall that a crossing point in a link diagram is \emph{nugatory} if there exists a $S^1$ on the plane which meets the link diagram only at this crossing point. Obviously for a nugatory crossing point on the plane, there are only three regions around it. Otherwise, locally the four regions around a crossing point would be distinct. Now if the link diagram $L_i$ contains no nugatory crossing point, then $M_{1\times c_i}$ is equal to the sum of all row vectors in $M_{L_i}'$. If there exist some nugatory crossing points in $L_i$, by using the algorithm in \cite{Shi2014} or \cite{Che2012} one can suitably choose some regions of $L_i$ such that the sum of the corresponding row vectors in $M_{L_i}'$ equals $M_{1\times c_i}$, which also follows that $\text{rank}_{\mathbb{Z}_2}(M_L)=\sum\limits_{i=1}^s\text{rank}_{\mathbb{Z}_2}(M_{L_i})$.
\end{remark}

\begin{corollary}\label{corollary1}
Let $L=K_1\cup\cdots\cup K_n$ be a link diagram on the plane, then each crossing point is region crossing change admissible if and only if each $K_i$ is isolated.
\end{corollary}
\begin{proof}
We use the same notation as above. According to Proposition \ref{proposition1}, each crossing point is region crossing change admissible if and only if $\text{rank}_{\mathbb{Z}_2}(M_L)=c$. Since $\text{rank}_{\mathbb{Z}_2}(M_L)=r-n-1$, it follows that $c=r-n-1=c+s+1-n-1=c+s-n$, which implies $s=n$.
\end{proof}

\begin{remark}
Proposition \ref{proposition2} provides a method to calculate the number of knot components of a link from the Tait graph and its dual. We want to remark that actually the Tait graph itself completely determines the component number. This is because the projection of the link diagram can be recovered from the Tait graph. Or, from another viewpoint, the dual graph depends on the embedding of the Tait graph. Due to Whitney's 2-isomorphism theorem \cite{Whi1933}, a pair of links corresponding to different embeddings of the same Tait graph can be connected by some mutations (e.g. \cite{CH2013}). Since mutation preserves the component number, it follows that the component number does not depend on the dual graph. Or, one can also consider the \emph{mod-2 Laplacian matrix} of the Tait graph. Denote the set of vertices of the Tait graph as $\{v_1, \cdots, v_m\}$, then the mod-2 Laplacian matrix is defined as $(a_{ij})_{m\times m}$, where $a_{ii}$ equals to the degree of $v_i$ and $a_{ij}$ equals to the number of edges between $v_i$ and $v_j$. As the name suggested, all entries take values in $\mathbb{Z}_2$. It was proved in \cite{God2001} that the number of components of the associated link is equal to the nullity of the matrix $(a_{ij})_{m\times m}$, which also means that the component number is determined by the Tait graph. A new self-contained proof of this result was given by Silver and Williams in \cite{Sil2015}.
\end{remark}

\section{Region crossing change on surfaces}\label{section3}
\subsection{Modified region crossing change}
Now we turn to study the behavior of region crossing changes on link diagrams on $\Sigma_g$. Here $\Sigma_g$ denotes a closed orientable surface of genus $g$. When the ambient space is changed from $S^2$ to $\Sigma_g$, there are two main differences:
\begin{enumerate}
\item For a given crossing point, locally there exist four regions around it. For $S^2$, there are two possibilities: these four regions are mutually distinct, or two opposite regions are part of the same region. In the second case, this crossing point must be nugatory. However, if the ambient space is $\Sigma_g$, one region can appears $i$ times around a crossing point, here $i$ runs over $\{0, 1, 2, 3, 4\}$. For example, consider the link diagram on $T^2$ (see Figure \ref{figure4}), the region $R$ appears $i$ times around $C_i$ ($0\leq i\leq 4$).
\item For a link diagram on $S^2$, it is well known that we can always color all the regions in checkerboard fashion. However, this is not the case for all link diagrams on $\Sigma_g$. For example, the link diagram in Figure \ref{figure4} does not admit a checkerboard fashion coloring. Note that the fact of existing a checkerboard fashion coloring plays an important role in \cite{Shi2014} and \cite{Das2018}.
\end{enumerate}

\begin{figure}[h]
\centering
\includegraphics{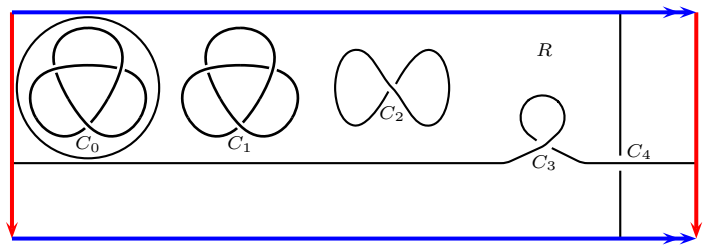}\\
\caption{A 6-component link diagram on $T^2$}\label{figure4}
\end{figure}

In order to overcome these difficulties, we make a little modification on the definition of region crossing change. Let $L$ be a link diagram on $\Sigma_g$ and $R$ a region of it. If $R$ appears $i$ times around a crossing point $C$, then taking region crossing change on $R$ will switch $C$ $i$ times. From now on when we say ``\emph{region crossing change}", it always refers to this modified version.

\begin{remark}
This is not the first time to study this kind of region crossing change in the literature. In \cite{Aha2012}, Kazushi Ahara and Masaaki Suzuki considered the \emph{single counting rule} and \emph{double counting rule}, which corresponds to the original region crossing change and our modified version (on $S^2$) respectively. For higher genus surface, this modified region crossing change was also studied by Dasbach and Russell in \cite{Das2018}, by using the Bollob\'{a}s-Riordan-Tutte polynomial. The different point is, in their paper they assumed that each region of a link diagram is a disk, i.e. the projection of the link diagram is cellularly embedded on $\Sigma_g$. Besides, they also assumed that each link diagram admits a checkerboard fashion coloring, which ensures the existence of Tait graph. Both assumptions are not needed in our paper.
\end{remark}

With this new version of region crossing change, for a given link diagram on $\Sigma_g$ one can similarly define the incidence matrix. The only different thing is, now if a region touches a crossing point $i$ times, then the corresponding entry equals $i$ (mod 2). For example, the new incidence matrix associated to the diagram in Figure \ref{figure3} has the form
\begin{center}
$M_K=\begin{pmatrix}
1&1&0&0&0\\
1&0&1&1&0\\
0&1&1&1&1\\
0&0&0&0&1\\
1&1&1&0&0\\
0&0&1&1&0\\
1&1&0&1&0
\end{pmatrix},$
\end{center}
since $R_7$ touches $C_5$ twice. It is routine to check that Proposition \ref{proposition1} still holds with respect to the modified region crossing change. In other words, in order to count the number of equivalence classes of link diagrams with the same projection on $\Sigma_g$, our main task is to figure out the $\mathbb{Z}_2$-rank of the incidence matrix.

\subsection{The graph $G_L$}
In order to have a better understanding of the effect of region crossing changes on a link diagram $L$ on $\Sigma_g$, it is convenient to introduce a new graph, which is denoted by $G_L$ in this paper.

Let $L$ be a link diagram on $\Sigma_g$ with crossing number $c$, then there are totally $2^c$ different link diagrams having the same projection as $L$. The graph $G_L$ consists of $2^c$ vertices, each one corresponds to one of these $2^c$ link diagrams. Now for each region, if two link diagrams are related by one region crossing change on this region then we add an edge connecting the corresponding two vertices. In particular, if a region touches every crossing point even times, then for each vertex of $G_L$ it adds a loop connecting this vertex to itself. As an example, consider the link projection on $T^2$ consisting of a meridian and a longitude. The corresponding graph has two components, each one consists of a vertex and an edge connecting the vertex to itself. As another example, Figure \ref{figure5} illustrates a trefoil knot diagram and its associated graph $G_K$.

\begin{figure}[h]
\centering
\includegraphics{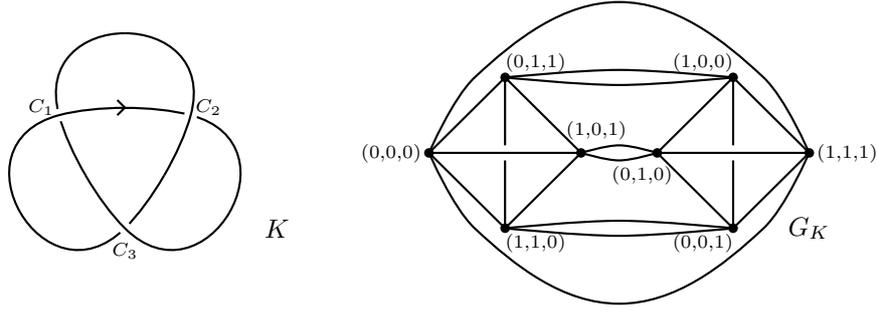}\\
\caption{A trefoil knot diagram $K$ and its associated graph $G_K$}\label{figure5}
\end{figure}

The graph $G_L$ contains a lot of information about region crossing changes. For example, the number of components of $G_L$ is equal to the number of equivalence classes of link diagrams with the same projection as $L$ under region crossing changes. In particular, each crossing point of $L$ is region crossing change admissible if and only if $G_L$ is connected. On the other hand, if a trail of $E_{G_L}$ forms a circuit then applying region crossing changes on the corresponding regions preserves every crossing point of $L$. The following proposition tells us that $G_L$ is highly symmetrical.

\begin{proposition}\label{proposition3}
Let $u$ and $v$ be two vertices of $G_L$, then there exists an isomorphism $h$ from $G_L$ to itself which brings $u$ to $v$.
\end{proposition}
\begin{proof}
Fix an orientation for each component of $L$, as mentioned in Section \ref{section2}, there is a one-to-one correspondence between the row vectors of $\mathbb{Z}_2^c$ and the vertex set of $G_L$. Furthermore, if two vertices are adjacent, then the difference of the corresponding row vectors is equal to the row vector corresponding to the edge connecting them. For the sake of simplicity, let us use the same symbol to denote a vertex in $G_L$, the corresponding link diagram and the associated row vector. Now the map $h: G_L\to G_L$ which sends $w\in V_{G_L}$ to $w+v-u$ is the desired automorphism of $G_L$. Notice that $w$ and $w'$ are adjacent if and only if $w+v-u$ and $w'+v-u$ are adjacent.
\end{proof}

\begin{corollary}\label{corollary2}
All connected components of $G_L$ are isomorphic.
\end{corollary}

It is an interesting question to ask which kind of graph can be realized as $G_L$ for some link diagram $L$ on $\Sigma_g$. Or for the sake of simplicity, we can restrict ourselves to knot diagrams on the plane. Note that in this case each graph is connected, see Corollary \ref{corollary3} below.

\subsection{Homotopy invariance}
The main aim of this paper is to count the number of components of $G_L$. Especially, we want to know when $G_L$ is connected. Proposition \ref{proposition1} tells us that the key ingredient is rank$_{\mathbb{Z}_2}(M_L)$. We end this section with an important property of rank$_{\mathbb{Z}_2}(M_L)$.

\begin{theorem}\label{theorem3}
Let $L$ be a link diagram on $\Sigma_g$ with $r$ regions, then $r-\text{rank}_{\mathbb{Z}_2}(M_L)$ is invariant under Reidemeister moves.
\end{theorem}
\begin{proof}
We check three Reidemeister moves one by one.
\begin{itemize}
\item The first Reidemeister move $\Omega_1$:
\begin{figure}[h]
\centering
\includegraphics{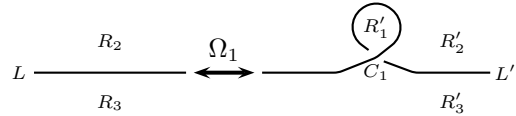}\\
\caption{Reidemeister move $\Omega_1$}\label{figure6}
\end{figure}

Suppose link diagrams $L$ and $L'$ are related by one $\Omega_1$. Denote the set of regions of $L$ and $L'$ by $\{R_2, R_3, \cdots, R_r\}$ and $\{R_1', R_2', R_3', \cdots, R_r'\}$ respectively, see Figure \ref{figure6}. Since $L'$ has one more region than $L$, it is sufficient to prove $\text{rank}_{\mathbb{Z}_2}(M_L)+1=\text{rank}_{\mathbb{Z}_2}(M_{L'})$. If $R_2$ and $R_3$ are different regions, then we have
\begin{center}
$\text{rank}_{\mathbb{Z}_2}(M_{L'})=\text{rank}_{\mathbb{Z}_2}\begin{pmatrix}
1&0&\cdots&0\\
0&&&\\
1&&&\\
0&&M_L&\\
\vdots&&&\\
0&&&
\end{pmatrix}=\text{rank}_{\mathbb{Z}_2}\begin{pmatrix}
1&0&\cdots&0\\
0&&&\\
0&&&\\
0&&M_L&\\
\vdots&&&\\
0&&&
\end{pmatrix}=\text{rank}_{\mathbb{Z}_2}(M_L)+1.$
\end{center}

If $R_2$ and $R_3$ are the same region, it suffices to combine the second row vector above with the third row vector. In this case, we still have $\text{rank}_{\mathbb{Z}_2}(M_{L'})=\text{rank}_{\mathbb{Z}_2}(M_L)+1$.

\item The second Reidemeister move $\Omega_2$:
\begin{figure}[h]
\centering
\includegraphics{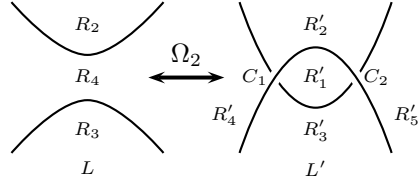}\\
\caption{Reidemeister move $\Omega_2$}\label{figure7}
\end{figure}

For the second Reidemeister move, see Figure \ref{figure7}, this time let us use $\{R_2, R_3, R_4\}$ and $\{R_1', R_2', R_3', R_4', R_5'\}$ to denote the regions of $L$ and $L'$ that involved in $\Omega_2$. We continue our discussion in two cases.
\begin{enumerate}
\item If $R_4'$ and $R_5'$ are distinct regions of $L'$, then $L'$ has two more regions than $L$, hence it suffices to prove that $\text{rank}_{\mathbb{Z}_2}(M_L)+2=\text{rank}_{\mathbb{Z}_2}(M_{L'})$. If $R_2', R_3', R_4', R_5'$ are mutually distinct regions, then we have
\begin{center}
$\text{rank}_{\mathbb{Z}_2}(M_{L'})=\text{rank}_{\mathbb{Z}_2}\begin{pmatrix}
1&1&0&\cdots&0\\
1&1&&r_2'&\\
1&1&&r_3'&\\
1&0&&r_4'&\\
0&1&&r_5'&\\
\vdots&\vdots&&&
\end{pmatrix}=\text{rank}_{\mathbb{Z}_2}\begin{pmatrix}
1&1&0&\cdots&0\\
0&0&&r_2'&\\
0&0&&r_3'&\\
0&0&&r_4'+r_5'&\\
0&1&&r_5'&\\
\vdots&\vdots&&&
\end{pmatrix}=\text{rank}_{\mathbb{Z}_2}\begin{pmatrix}
1&1&0&\cdots&0\\
0&1&&r_5'&\\
0&0&&&\\
0&0&&M_L&\\
0&0&&&\\
\vdots&\vdots&&&
\end{pmatrix}=\text{rank}_{\mathbb{Z}_2}(M_L)+2$.
\end{center}
If some of $\{R_2', R_3', R_4', R_5'\}$ are parts of the same region, it suffices to combine the corresponding row vectors into one row vector. For example, if $R_3'$ and $R_5'$ are actually the same region (hence $R_3$ and $R_4$ also denote the same region), let us still use $R_3'$ to denote this region. One can image this by adding a tube connecting the two regions $R_3'$ and $R_5'$ on the plane. In this case, it suffices to combine $R_4'$ with $R_3'$ and one computes
\begin{center}
$\text{rank}_{\mathbb{Z}_2}(M_{L'})=\text{rank}_{\mathbb{Z}_2}\begin{pmatrix}
1&1&0&\cdots&0\\
1&1&&r_2'&\\
1&0&&r_3'&\\
1&0&&r_4'&\\
\vdots&\vdots&&&
\end{pmatrix}=\text{rank}_{\mathbb{Z}_2}\begin{pmatrix}
1&1&0&\cdots&0\\
1&0&&r_3'&\\
0&0&&&\\
0&0&&M_L&\\
\vdots&\vdots&&&
\end{pmatrix}=\text{rank}_{\mathbb{Z}_2}(M_L)+2$.
\end{center}
Other cases can be verified in a similar manner.

\item If $R_4'$ and $R_5'$ are parts of the same region of $L'$ (still denoted by $R_4'$), in this case $L'$ has only one more region than $L$. Then we are needed to show that $\text{rank}_{\mathbb{Z}_2}(M_L)+1=\text{rank}_{\mathbb{Z}_2}(M_{L'})$. Similarly, let us first consider the simplest case that $R_2', R_3', R_4'$ are mutually distinct regions, then
\begin{center}
$\text{rank}_{\mathbb{Z}_2}(M_{L'})=\text{rank}_{\mathbb{Z}_2}\begin{pmatrix}
1&1&0&\cdots&0\\
1&1&&r_2'&\\
1&1&&r_3'&\\
1&1&&r_4'&\\
\vdots&\vdots&&&
\end{pmatrix}=\text{rank}_{\mathbb{Z}_2}\begin{pmatrix}
1&1&0&\cdots&0\\
0&0&&r_2'&\\
0&0&&r_3'&\\
0&0&&r_4'&\\
\vdots&\vdots&&&
\end{pmatrix}=\text{rank}_{\mathbb{Z}_2}\begin{pmatrix}
1&1&0&\cdots&0\\
0&0&&&\\
0&0&&M_L&\\
0&0&&&\\
\vdots&\vdots&&&
\end{pmatrix}=\text{rank}_{\mathbb{Z}_2}(M_L)+1$.
\end{center}
Other cases can be verified similarly as above. We omit the details here.
\end{enumerate}

\item The third Reidemeister move $\Omega_3$:
\begin{figure}[h]
\centering
\includegraphics{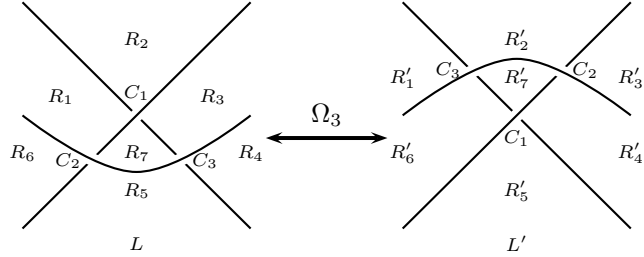}\\
\caption{Reidemeister move $\Omega_3$}\label{figure8}
\end{figure}

Assume $L$ and $L'$ are two link diagrams which are related by one $\Omega_3$ (see Figure \ref{figure8}), let us use $\{R_1, \cdots, R_r\}$ and $\{R_1', \cdots, R_r'\}$ to denote the regions of $L$ and $L'$ respectively. We only check the case that $R_1, \cdots, R_6$ (hence also $R_1', \cdots, R_6'$) are mutually different. Other cases can be proved in a similar manner.

Notice that $M_L$ and $M_{L'}$ have the following forms
\begin{center}
$M_L=\begin{pmatrix}
1&1&0&&r_1&\\
1&0&0&&r_2&\\
1&0&1&&r_3&\\
0&0&1&&r_4&\\
0&1&1&&r_5&\\
0&1&0&&r_6&\\
1&1&1&0&\cdots&0\\
\vdots&\vdots&\vdots&&M&
\end{pmatrix}$ and $M_{L'}=\begin{pmatrix}
0&0&1&&r_1&\\
0&1&1&&r_2&\\
0&1&0&&r_3&\\
1&1&0&&r_4&\\
1&0&0&&r_5&\\
1&0&1&&r_6&\\
1&1&1&0&\cdots&0\\
\vdots&\vdots&\vdots&&M&
\end{pmatrix}$.
\end{center}
Here each $r_i$ ($1\leq i\leq 6$) denotes a row vector of $\mathbb{Z}_2^{c-3}$, and $M$ denotes a $(r-7)\times(c-3)$ matrix. As before, we use $r$ and $c$ to denote the number of regions and the number of crossing respectively. Since there is a one-to-one correspondence between the regions of $L$ ($L'$) and row vectors of $M_L$ ($M_{L'}$), we will use the same notation to denote a region and the corresponding row vector. Notice that $R_i+R_7=R_i'$ ($1\leq i\leq 6$) and $R_7=R_7'$, it follows immediately that $\text{rank}_{\mathbb{Z}_2}(M_L)=\text{rank}_{\mathbb{Z}_2}(M_{L'})$. Together with the fact that $L$ and $L'$ have the same number of regions. The proof is finished.
\end{itemize}
\end{proof}

\begin{corollary}\label{corollary3}
With the modified region crossing change, for each $n$-component link diagram $L$ on the plane, we have rank$_{\mathbb{Z}_2}(M_L)=r-n-1$. In particular, each crossing point of a knot diagram on the plane is region crossing change admissible.
\end{corollary}
\begin{proof}
Note that rank$_{\mathbb{Z}_2}(M_L)$ only depends on the projection of $L$, and each $n$-component link projection is homotopy equivalent to the disjoint union of $n$ simple closed curves.
\end{proof}

\begin{remark}
Modified region crossing change is an unknotting operation for knot diagrams on the plane is not a new result. Actually, the proof of the fact that the region crossing change with the double counting rule is an unknotting operation for knot diagrams on the plane is much simpler than the original one. The reader can obtain this result from Ayaka Shimizu's original proof \cite{Shi2014}. Later in \cite{Aha2012}, this result was reproved and extended by Ahara and Suzuki. Recently, Dasbach and Russell proved that the component number of $G_L$ for a link diagram $L$ on the plane is equal to the absolute value of the Tutte polynomial of the associated Tait graph evaluated at $(-1, -1)$, which is proved to be $2^{n-1}$ \cite{Das2018}. As a consequence, we have $c-\text{rank}_{\mathbb{Z}_2}(M_L)=n-1$. Together with the fact $r=c+2$, it also leads to the result rank$_{\mathbb{Z}_2}(M_L)=r-n-1$.
\end{remark}

\begin{proposition}\label{proposition4}
Let $K$ be a knot diagram on $T^2$, if the homotopy class $[K]=p[m]+q[l]$ where $[m], [l]$ denote the canonical generators of $\pi_1(T^2)$, then rank$_{\mathbb{Z}_2}(M_K)=r-2$ if $k$ is even, otherwise rank$_{\mathbb{Z}_2}(M_K)=r-1$. Here $k=gcd(p, q)$.
\end{proposition}
\begin{proof}
According to Theorem \ref{theorem3}, we can choose a minimal diagram $K'$ such that $[K']=[K]=p[m]+q[l]$. If $p=q=0$, then $k=0$ and $K'$ is a simple closed curve on $T^2$ which bounds a disk. In this case the result follows obviously. If gcd$(p, q)=k>0$, it is not difficult to observe that $c(K')=k-1$ and $r(K')=k$. See the left side of Figure \ref{figure9} for an example of $K'$, where $[K']=12[m]+8[l]$.

\begin{figure}[h]
\centering
\includegraphics{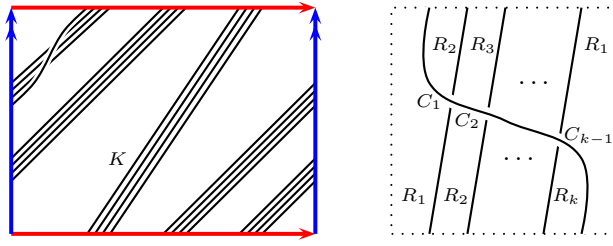}\\
\caption{A knot diagram on $T^2$}\label{figure9}
\end{figure}

Consider a small neighborhood of the crossing points $C_1, \cdots, C_{k-1}$, and denote all the regions by $R_1, \cdots, R_{k}$, see the right side of Figure \ref{figure9}. Then the incidence matrix $M_K$ has the form
\begin{center}
$M_K=\begin{pmatrix}
1&0&0&0&\cdots&0&1\\
0&1&0&0&\cdots&0&0\\
1&0&1&0&0&\cdots&0\\
0&1&0&1&0&\cdots&0\\
0&0&1&0&1&\cdots&0\\
\vdots&\vdots&&\ddots&&\ddots&\\
0&0&\cdots&0&1&0&1\\
0&0&\cdots&0&0&1&0\\
\end{pmatrix}$.
\end{center}
When $k$ is even, it is easy to observe that $R_2, \cdots, R_{k-1}$ form a basis for the row space. If $k$ is odd, then $R_2, \cdots, R_k$ form a basis for the row space. As before, here we also use $R_i$ to refer to the corresponding row vector of $M_K$. It follows that $r-\text{rank}_{\mathbb{Z}_2}(M_K)=k-\text{rank}_{\mathbb{Z}_2}(M_{K'})=k-(k-2)=2$ if $k$ is even; otherwise $r-\text{rank}_{\mathbb{Z}_2}(M_K)=k-\text{rank}_{\mathbb{Z}_2}(M_{K'})=k-(k-1)=1$. The proof is completed.
\end{proof}

\begin{corollary}\label{corollary4}
Let $K$ be a knot diagram on $T^2$ which satisfies $[K]=p[m]+q[l]\in\pi_1(T^2)$, then
\begin{itemize}
\item if gcd$(p, q)$ is even, $G_K$ is connected if and only if there exists a disk $D$ in $T^2$ such that $K\subset D$;
\item if gcd$(p, q)$ is odd, $G_K$ is connected if and only if there exists an annulus $A$ in $T^2$ such that $K\subset A$.
\end{itemize}
\end{corollary}
\begin{proof}
We divide the proof into two cases.

\begin{itemize}
\item If gcd$(p, q)$ is even, according to Proposition \ref{proposition4} we know that rank$_{\mathbb{Z}_2}(M_K)=r-2$. On the other hand, Proposition \ref{proposition1} tells us that $G_K$ is connected if and only if rank$_{\mathbb{Z}_2}(M_K)=c$, which follows that $r-2=c$. Denote all the regions by $R_1, \cdots, R_r$, then we have $c-2c+\sum\limits_{i=1}^r\chi(R_i)=\chi(T^2)=0$, where $\chi(R_i)$ and $\chi(T^2)$ denote the Euler characteristic of $R_i$ and $T^2$ respectively. Therefore $\sum\limits_{i=1}^r\chi(R_i)=c=r-2$. There are two possibilities:
\begin{enumerate}
\item Two regions, say $R_1$ and $R_2$, have Euler characteristic zero and all other regions have Euler characteristic one. Then $R_1$ and $R_2$ are both homeomorphic to an open annulus. Due to the Euler characteristic reason, it is not difficult to observe that $T^2\setminus(R_1\cup R_2)$ is disconnected. This contradicts with the fact that $K$ is connected.
\item One region, say $R_1$, has Euler characteristic minus one and all other regions have Euler characteristic one. Then $R_1$ is either homeomorphic to a 2-sphere with three disks removed or a torus with one disk removed. The first case also contradicts with the fact $K$ is connected. Therefore $R_1$ is homeomorphic to $T^2$ with a disk moved. This disk is the disk desired.
\end{enumerate}
Conversely, if $K$ is bounded in a disk, then it follows from Corollary \ref{corollary3} that each crossing point is region crossing change admissible. Hence $G_K$ is connected.
\item If gcd$(p, q)$ is odd, in this case $G_K$ is connected if and only if $c=r-1$. With the same notations as above, now we have $\sum\limits_{i=1}^r\chi(R_i)=c=r-1$. The only possibility is one region $R_1$ has Euler characteristic zero and all other regions have Euler characteristic one. Hence $R_1$ is homeomorphic to an open annulus and all other regions are homeomorphic to open disks. Note that the core of $A$ represents a nontrivial homotopy element of $\pi_1(T^2)$, otherwise $p=q=0$, which contradicts with the assumption that gcd$(p, q)$ is an odd integer. Now a small neighborhood of the complement of $R_1$ provides the desired annulus $A$.

Conversely, assume there exists an annulus $A\subset T^2$ such that the knot diagram $K\subset A$. According to the assumption that gcd$(p, q)$ is an odd integer, it suffices to show the equation $c=r-1$ holds. Notice that $[K]\neq0\in\pi_1(A)$ and $T^2\setminus A$ is connected, since gcd$(p, q)$ is odd. Denote all the regions of $T^2\setminus K$ by $R_1, R_2, \cdots, R_r$ such that $R_i\subset A$ ($2\leq i\leq r$). It is not difficult to observe that each $R_i$ ($2\leq i\leq r$) is homeomorphic to an open disk. Then we have
\begin{center}
$0=c-2c+\sum\limits_{i=1}^r\chi(R_i)=-c+0+\sum\limits_{i=2}^r\chi(R_i)=r-1-c$,
\end{center}
then we obtain the result desired.
\end{itemize}
\end{proof}

\section{The main result}\label{section4}
Theorem \ref{theorem3} is extremely useful for calculating rank$_{\mathbb{Z}_2}(M_K)$ when the genus of the surface is small or the homotopy class $[K]\in\pi_1(\Sigma_g)$ is simple. However, if a minimal link diagram $L'$ with $[L']=[L]\in\pi_1(\Sigma_g)$ is still very complicated, it is not an easy job to read rank$_{\mathbb{Z}_2}(M_L)$ directly from $L'$. The main aim of this section is to provide a method to calculate rank$_{\mathbb{Z}_2}(M_L)$ directly from the link diagram $L$. In particular, Theorem \ref{theorem3} tells us that $r-\text{rank}_{\mathbb{Z}_2}(M_L)$ is a homotopy invariant. We will find from the following theorem that actually it only depends on the homology class of $L$.

Let $L=K_1\cup\cdots\cup K_n$ be a link diagram on $\Sigma_g$ and $\{\alpha_1, \cdots, \alpha_{2g}\}$ a canonical basis of $H_1(\Sigma_g, \mathbb{Z}_2)$. Assume each homology class $[K_i]\in H_1(\Sigma_g, \mathbb{Z}_2)$ can be written as $[K_i]=\sum\limits_{j=1}^{2g}b_{ij}\alpha_j$, then we have a new matrix $N_L=(b_{ij})_{n\times 2g}$. The following theorem extends the result of Corollary \ref{corollary3} from $S^2$ to $\Sigma_g$.

\begin{theorem}\label{theorem4}
For a given link diagram $L=K_1\cup\cdots\cup K_n$ on $\Sigma_g$, we have $\text{rank}_{\mathbb{Z}_2}(M_L)=r-n-1+\text{rank}_{\mathbb{Z}_2}(N_L)$. Here $r$ denotes the number of regions.
\end{theorem}

Before proving Theorem \ref{theorem4}, we need two lemmas. For a given link diagram $L$ on $\Sigma_g$, if all the regions of $L$ admits a checkerboard fashion coloring, i.e. each region is colored white or black such that locally any two adjacent regions receive distinct colors, then we say $L$ is \emph{2-colorable}. For example, a meridian on $T^2$ is not 2-colorable, since the unique region is adjacent to itself.

\begin{lemma}\label{lemma1}
A link diagram $L=K_1\cup\cdots\cup K_n$ on $\Sigma_g$ is 2-colorable if and only if $\sum\limits_{i=1}^n[K_i]=0\in H_1(\Sigma_g, \mathbb{Z}_2)$.
\end{lemma}
\begin{proof}
Denote all the regions of $L$ by $R_1, \cdots, R_r$. If $\sum\limits_{i=1}^n[K_i]=0\in H_1(\Sigma_g, \mathbb{Z}_2)$, then there exist some regions, say $R_1, \cdots, R_s$, such that $\partial R_1\cup\cdots\cup\partial R_s=K_1\cup\cdots\cup K_n$. Note that here we count with $\mathbb{Z}_2$-coefficient. However, locally one arc of $L$ is adjacent to at most two regions, it follows that if two regions are locally adjacent then only one of them belongs to the set $\{R_1, \cdots, R_s\}$. Coloring regions $R_1, \cdots, R_s$ white and all other regions black yields the 2-coloring desired.

Conversely, if $L$ is 2-colorable, let us use $R_1, \cdots, R_s$ to denote all the regions colored white. If $L$ has no crossing point, i.e. each $K_i$ is a simple closed curve and $K_i\cap K_j=\emptyset$ ($1\leq i<j\leq n$), then $\partial R_1\cup\cdots\cup\partial R_s=K_1\cup\cdots\cup K_n$. With a suitable choice of orientation for each $K_i$, we have $\sum\limits_{i=1}^n[K_i]=0\in H_1(\Sigma_g, \mathbb{Z})$, which follows that $\sum\limits_{i=1}^n[K_i]=0\in H_1(\Sigma_g, \mathbb{Z}_2)$. If the crossing number of $L$ is positive, choose a crossing point and smooth it such that the number of components of $L$ is preserved. Denote the new link diagram by $L'=K_1'\cup\cdots\cup K_n'$, then we have $\sum\limits_{i=1}^n[K_i]=\sum\limits_{i=1}^n[K_i']\in H_1(\Sigma_g, \mathbb{Z}_2)$. Notice that $L$ is 2-colorable if and only if $L'$ is 2-colorable. Continue this process until there is no crossing points, according to the discussion above one concludes that $\sum\limits_{i=1}^n[K_i]=0\in H_1(\Sigma_g, \mathbb{Z}_2)$.
\end{proof}

Recall that there is a one-to-one correspondence between regions (crossing points) and row (column) vectors of the incidence matrix. Similar as before, we will use the same symbol $R_i$ ($C_i$) to refer to a region (crossing point) and the corresponding row (column) vector.

\begin{lemma}\label{lemma2}
Let $L=K_1\cup\cdots\cup K_n$ be a link diagram on $\Sigma_g$ and $\{R_1, \cdots, R_r\}$ the set of all regions. If $\sum\limits_{i=1}^sR_i=0$ ($s<r$), then there exists a sub-link $L'$ such that $L'$ is 2-colorable. Conversely, if a sub-link of $L$ is 2-colorable then there exist some regions such that the sum of them is equal to the zero vector.
\end{lemma}
\begin{proof}
First notice that if $\sum\limits_{i=1}^sR_i=0$ then $\sum\limits_{i=s+1}^rR_i=0$, since $\sum\limits_{i=1}^rR_i=0$. Dye the regions $R_1, \cdots, R_s$ white and dye $R_{s+1}, \cdots, R_r$ black. A key observation is, if two regions adjacent to the same arc of $K_i$ have different colors, then along $K_i$ ($1\leq i\leq n$), any pair of regions adjacent to the same arc of $K_i$ also have different colors. Consequently, if two regions on the two sides of $K_i$ ($1\leq i\leq n$) have the same color then any pair of regions on the two sides of $K_i$ also have the same color. In order to see this, consider two regions, say $R_1$ and $R_2$, are adjacent to the same arc of $K_i$. According to our coloring rule above, $R_1$ and $R_2$ are both colored white. Walking along $K_i$, when we meet the first crossing point the two new regions must have the same color, since $\sum\limits_{i=1}^sR_i=0$. Repeating this deduction until we come back to our beginning point, one observes that if two regions located in the two sides of the same arc of $K_i$ respectively, then  just like $R_1$ and $R_2$, they also have the same color. Figure \ref{figure10} provides a diagrammatic sketch of this deduction.

\begin{figure}[h]
\centering
\includegraphics{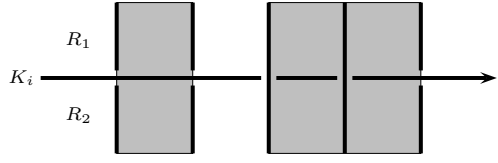}\\
\caption{Coloring regions along $K_i$}\label{figure10}
\end{figure}

Without loss of generality, we assume $K_1, \cdots, K_t$ are those knot components which satisfy that any two regions adjacent to the same arc of some $K_i$ ($1\leq i\leq t$) have different colors. Then our coloring rule provides a checkerboard fashion coloring for the sub-link $K_1\cup\cdots\cup K_t$.

Conversely, if a sub-link $L'$ is 2-colorable, notice that a region of $L'$ is a disjoint union of some regions of $L$, then the sum of all white regions equals the zero vector. The proof is finished.
\end{proof}

Now we turn to the proof of Theorem \ref{theorem4}.
\begin{proof}
The proof consists of two parts:
\begin{itemize}
\item $r-\text{rank}_{\mathbb{Z}_2}(M_L)\geq n+1-\text{rank}_{\mathbb{Z}_2}(N_L)$.

Assume $n-\text{rank}_{\mathbb{Z}_2}(N_L)=k$ and $\{[K_{k+1}], \cdots, [K_{n}]\}$ is a basis for the row space of $N_L$. Here $[K_i]$ is not only referred to the homology class of $K_i$ in $H_1(\Sigma_g, \mathbb{Z}_2)$, but also the corresponding row vector in $N_L$. Then each $[K_i]$ ($1\leq i\leq k$) can be written as the sum of some elements in $\{[K_{k+1}], \cdots, [K_n]\}$. According to Lemma \ref{lemma1}, now we have $k$ different 2-colorable sub-links of $L$. Lemma \ref{lemma2} tells us that each one of these 2-colorable sub-links gives rise to some regions whose sum is equal to zero. Let us write them down as below
\begin{center}
$\sum\limits_{R_i\in A_1}R_i=0, \cdots, \sum\limits_{R_i\in A_k}R_i=0$,
\end{center}
here each $A_i$ ($1\leq i\leq k$) is a subset of $\{R_1, \cdots, R_r\}$. Together with $\sum\limits_{i=1}^rR_i=0$, now we have $k+1$ linearly dependent sets of row vectors. In order to prove $r-\text{rank}_{\mathbb{Z}_2}(M_L)\geq k+1$, it suffices to show that the coefficients of these $k+1$ equalities are linearly independent, i.e. any equality of these can not be derived from the rest $k$ equalities.

Consider an arc of $K_1$, denote the two regions adjacent to it by $R_1$ and $R_2$. According to the method of choosing regions discussed in Lemma \ref{lemma2}, only one of $R_1, R_2$ belongs to $A_1$, and either $R_1\in A_i, R_2\in A_i$ or $R_1\notin A_i, R_2\notin A_i$ if $i\geq 2$. It follows that any equality of $\sum\limits_{i\in A_1}R_i=0, \cdots, \sum\limits_{i\in A_k}R_i=0$ can not be obtained from the others. Moreover, we claim the last equality $\sum\limits_{i=1}^rR_i=0$ also can not be derived from the first $k$ equalities. For example, assume $R_1\in A_1$ but $R_2\notin A_1$, since either $R_1\in A_i, R_2\in A_i$ or $R_1\notin A_i, R_2\notin A_i$ for all $2\leq i\leq k$, one deduces that the first equality $\sum\limits_{R_i\in A_1}R_i=0$ can not be used to obtain the equality $\sum\limits_{i=1}^rR_i=0$. Analogously, one can prove a similar result for any $2\leq i\leq k$. Therefore we conclude that $r-\text{rank}_{\mathbb{Z}_2}(M_L)\geq k+1=n+1-\text{rank}_{\mathbb{Z}_2}(N_L)$.

\item $r-\text{rank}_{\mathbb{Z}_2}(M_L)\leq n+1-\text{rank}_{\mathbb{Z}_2}(N_L)$.

This time we assume $r-\text{rank}_{\mathbb{Z}_2}(M_L)=k$ and denote a basis for the row space of $M_L$ by $R_{k+1}, \cdots, R_r$. Now each $R_i$ ($1\leq i\leq k$) can be written as a linear combination of some elements in $\{R_{k+1}, \cdots, R_r\}$. According to Lemma \ref{lemma2}, these provides us with $k$ 2-colorable sub-links, denoted by $\cup_{K_i\in B_1}K_i, \cdots, \cup_{K_i\in B_k}K_i$. Here $B_i$ ($1\leq i\leq k$) is a subset of $\{K_1, \cdots, K_n\}$. Together with Lemma \ref{lemma1}, one obtains that
\begin{center}
$\sum\limits_{i\in B_1}[K_i]=0, \cdots, \sum\limits_{i\in B_{k-1}}[K_i]=0$.
\end{center}
The reason why we drop the last equality is, the last equality $\sum\limits_{i\in B_k}[K_i]=0$ can be obtained from the first $k-1$ equalities. Actually, the fact that $R_k$ can be written as the sum of some elements of $\{R_{k+1}, \cdots, R_r\}$ can be replaced by the fact that $\sum\limits_{i=1}^rR_i=0$, since $r-\text{rank}_{\mathbb{Z}_2}(M_L)=k$. However, the assumption of Lemma \ref{lemma2} requires the sum of a proper subset of $\{R_1, \cdots, R_r\}$ equals zero.

We next demonstrate that any one of $\sum\limits_{i\in B_1}[K_i]=0, \cdots, \sum\limits_{i\in B_{k-1}}[K_i]=0$ can not be derived from the rest $k-2$ equalities, which follows that $n-\text{rank}_{\mathbb{Z}_2}(N_L)\geq k-1=r-\text{rank}_{\mathbb{Z}_2}(M_L)-1$. Suppose, to the contrary of the conclusion, that one equality can be derived from some others. Without loss of generality, let us assume $\sum\limits_{i\in B_1}[K_i]+\cdots+\sum\limits_{i\in B_j}[K_i]=0$ for some integer $j\leq k-1$. Recall that each equality $\sum\limits_{i\in B_1}[K_i]=0$ corresponds to some regions, and the sum of them is equal to zero. According to our method of choosing knot components (see the proof of Lemma \ref{lemma2}), if we put all the regions corresponding to $\sum\limits_{i\in B_1}[K_i]=0, \cdots, \sum\limits_{i\in B_j}[K_i]=0$ together, counted with multiplicity, it is not difficult to observe that if two regions are adjacent then the number of times they appear in this multiset have the same parity. It immediately follows that the multiplicity of each $R_i$ ($1\leq i\leq r$) in this multiset has the same parity. However, we know that $R_1$ appears once in this multiset but $R_k$ does not appear, which is a contradiction.
\end{itemize}
\end{proof}

\begin{example}
We use one example to explain how to use Theorem \ref{theorem4} to calculate the number of equivalence classes of link diagrams under region crossing changes. Consider the 4-component link diagram on a torus, see Figure \ref{figure11}. The matrix $N_L$ has the form
\begin{center}
$N_L=\begin{pmatrix}
1&0\\
1&0\\
0&1\\
0&1
\end{pmatrix}$,
\end{center}
hence rank$_{\mathbb{Z}_2}(N_L)=2$. According to Theorem \ref{theorem4}, we obtain rank$_{\mathbb{Z}_2}(M_L)=8-(4+1-2)=5$. It follows that $G_L$ has $2^{8-5}=8$ connected components.
\begin{figure}
\centering
\includegraphics{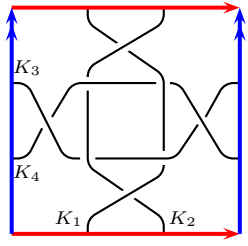}\\
\caption{A 4-component link diagram on $T^2$}\label{figure11}
\end{figure}
\end{example}

\begin{remark}
We would like to remark that actually the example above comes from \cite{Das2018}. In that article, Dasbach and Russell studied the number of equivalence classes, or equivalently, the number of components of $G_L$ via the Tait graph $T_L$ (which is required to be connected there) and its dual graph $T_L'$. A formula was given in \cite{Das2018} to calculate the number of components of $G_L$, which states that it equals $2^{2g+\text{dim ker}(\varphi|_{\mathcal{P}})}$. Here $\varphi$ denotes a map from $U^{\perp}$ to $H_1(\Sigma_g, \mathbb{Z}_2)$, where $U$, a subspace of $2^c$, is generated by the row vectors of $M_{T_L}$, and $\mathcal{P}$, a subspace of $U^{\perp}$, is generated by some row vectors corresponding to knot components. The reader is referred to \cite{Das2018} for more details. For this example, Dasbach and Russell showed that dim ker$(\varphi|_{\mathcal{P}})=1$, therefore the result coincides with our result above.
\end{remark}

\begin{example}
As another illustration, let us consider the knot diagram $K$ in Figure \ref{figure12}, where $[K]=p[m]+0[l]$ and $p$ is an odd integer. Notice that this diagram does not admit a checkerboard fashion coloring, and there does not exist a Tait graph associated to $K$. Therefore the formula mentioned above can not apply. Since $p$ is odd, one observes that rank$_{\mathbb{Z}_2}(N_L)=1$ and hence rank$_{\mathbb{Z}_2}(M_L)=p-(1+1-1)=p-1$. It follows that $G_K$ is connected. Note that this result also can be obtained from Corollary \ref{corollary4}.

\begin{figure}[h]
\centering
\includegraphics{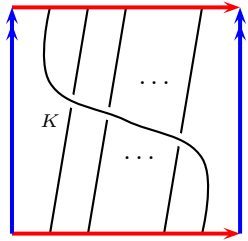}\\
\caption{A knot diagram on $T^2$}\label{figure12}
\end{figure}
\end{example}

\begin{corollary}\label{corollary5}
Let $K$ be a knot diagram on $\Sigma_g$ such that the projection is a cellularly embedded 4-valent graph, then $G_K$ is connected if and only if $g=0$.
\end{corollary}
\begin{proof}
If $\Sigma_g=S^2$, the results follows from Corollary \ref{corollary3}. Conversely, if the projection of $K$ is cellularly embedded, then each region is homeomorphic to an open disk. It follows that $c-2c+r=2-2g$. We know that $G_K$ is connected if and only if rank$_{\mathbb{Z}_2}(M_K)=c$. According to Theorem \ref{theorem4}, $r-\text{rank}_{\mathbb{Z}_2}(M_K)=n+1-\text{rank}_{\mathbb{Z}_2}(N_K)$, which implies $2-2g+c-c=1+1-\text{rank}_{\mathbb{Z}_2}(N_K)$. Hence we conclude that $G_K$ is connected if and only if $\text{rank}_{\mathbb{Z}_2}(N_K)=2g$. Since the size of $N_K$ is $1\times2g$, it follows that $g=0$.
\end{proof}

\section{The original region crossing change revisited}\label{section5}
In the end of this paper, let us go back to the original region crossing change, i.e. the $(i, j)$th entry $m_{ij}$ of the incidence matrix $M_L=(m_{ij})_{r\times c}$ is equal to 1 if the crossing point $C_j$ is on the boundary of the region $R_i$, otherwise $m_{ij}=0$. With this setting, Theorem \ref{theorem4} is not valid anymore. Actually, Lemma \ref{lemma2} is no longer true. In order to see this, let us consider the knot diagram $K$ depicted in Figure \ref{figure12}. But this time, $p$ is allowed to run over all nonnegative integers. Now the matrix $M_K$ has the following form
\begin{center}
$M_K=\begin{pmatrix}
1&&&&&&1\\
1&1&&&&&\\
1&1&1&&&&\\
&1&1&1&&&\\
&&\ddots&\ddots&\ddots&&\\
&&&1&1&1&\\
&&&&1&1&1\\
&&&&&1&1
\end{pmatrix}_{p\times(p-1)}$.
\end{center}
Direct calculation shows that rank$_{\mathbb{Z}_2}(M_K)=p-2=r-2$ if $p$ is divisible by 3, otherwise rank$_{\mathbb{Z}_2}(M_K)=p-1=r-1$. In particular, if $p$ is not a multiple of 3 then the graph $G_K$ is connected, or equivalently speaking, each crossing point is region crossing change admissible. However, we know that rank$_{\mathbb{Z}_2}(N_K)=0$ if $p$ is even, and rank$_{\mathbb{Z}_2}(N_K)=1$ if $p$ is odd. Hence
\begin{itemize}
\item $r-\text{rank}_{\mathbb{Z}_2}(M_K)=1<2=1+1-\text{rank}_{\mathbb{Z}_2}(N_K)$, if $p=2$;
\item $r-\text{rank}_{\mathbb{Z}_2}(M_K)=2>1=1+1-\text{rank}_{\mathbb{Z}_2}(N_K)$, if $p=3$.
\end{itemize}
The relationship among rank$_{\mathbb{Z}_2}(M_K)$, rank$_{\mathbb{Z}_2}(N_K)$ and the number of regions seems a bit mysterious. We end this paper with a lower bound for rank$_{\mathbb{Z}_2}(M_L)$, with respect to the original region crossing change.

\begin{proposition}\label{proposition5}
Let $L$ be a $n$-component link diagram on $\Sigma_g$, then rank$_{\mathbb{Z}_2}(M_L)\geq r-n-1$.
\end{proposition}
\begin{proof}
$\Sigma_g$ can be constructed from a regular $4g$-gon by identifying pairs of edges. For each pair of edges needed to be identified, one can add some parallel curves outside of the polygon, connecting the intersection points between $L$ and these two edges. Now we obtain a new link diagram $L'$ on the plane, which has the same number of components as $L$. See Figure \ref{figure13} for an example. On the left hand side we have a 3-component link diagram on $\Sigma_2$ with 17 crossing points, on the right hand side after adding some curves we obtain a 3-component link diagram on the plane, which now has 29 crossing points. For simplicity, we only draw the projections and ignore the crossing information.

\begin{figure}[h]
\centering
\includegraphics{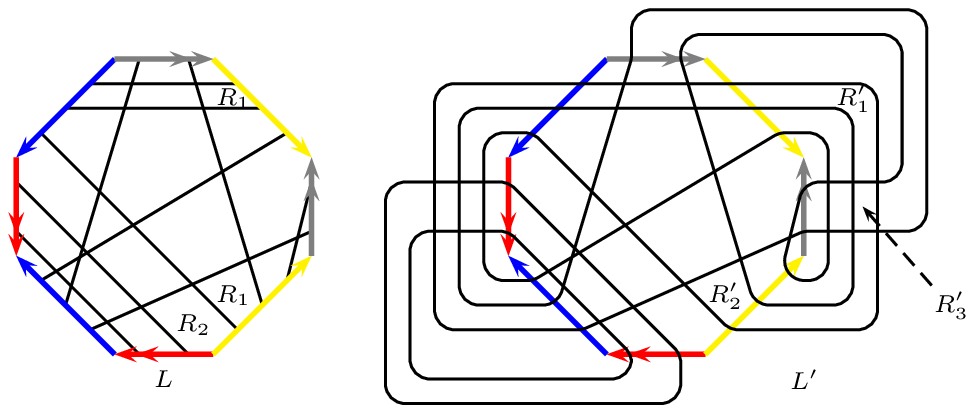}\\
\caption{From a link diagram on $\Sigma_g$ to a link diagram on $R^2$}\label{figure13}
\end{figure}

Let us use $R_1, \cdots, R_{r(L)}$ and $R_1', \cdots, R_{r(L')}'$ to denote the regions of $\Sigma_g\setminus L$ and $R^2\setminus L'$ respectively. Now the set $\{R_1', \cdots, R_{r(L')}'\}$ can be divided into the disjoint union of three subsets $\{R_1', \cdots, R_{r(L')}'\}=A\cup B\cup C$, where $A$ includes the regions that have nonempty intersection with the interior of the polygon except those which contain some vertices of the polygon, $B$ denotes the set of regions which contain some vertices of the polygon, and the rest regions constitute the subset $C$. Notice that the unique unbounded region belongs to the subset $B$ and each region of $C$ is sharped as a rectangle. The following are two noteworthy facts:
\begin{itemize}
\item If a region of $L$ meets the boundary of the polygon, but does not meet any vertex of the polygon, then there exist a finite sequence of regions of $L'$ such that applying region crossing change on them have the same effect. For example, in Figure \ref{figure13} one observes that applying region crossing change on $R_1$ of $L$ and $R_1'\cup R_2'\cup R_3'$ of $L'$ have the same effect.
\item Among the regions $R_1, \cdots, R_r$, there exists only one region which contains all the vertices of the polygon. In Figure \ref{figure13}, we use $R_2$ to denote it. Notice that applying region crossing change on this region is equivalent to take region crossing change on all the regions in $B\cup C$.
\end{itemize}
The two observations above imply that after some elementary row transformations the matrix $M_{L'}$ has the form
\begin{center}
$\begin{pmatrix}
M_L&0\\
\ast&\ast
\end{pmatrix}$.
\end{center}
Since rank$_{\mathbb{Z}_2}(M_{L'})=r(L')-n-1$ (Proposition \ref{proposition2}), it follows rank$_{\mathbb{Z}_2}(M_L)\geq r(L)-n-1$.
\end{proof}

\begin{remark}
In fact, the proof above also provides a upper bound for rank$_{\mathbb{Z}_2}(M_L)$. A key observation is, all the row vectors in $M_{L'}$ corresponding to the regions in $C$ are actually linearly independent. This can be verified directly on the diagram, note that each region in $C$ is a rectangle. Then it follows that
\begin{center}
rank$_{\mathbb{Z}_2}(M_L)\leq |A|+|B|-n-1$.
\end{center}
Here $|A|$ and $|B|$ denote the cardinality of $A$ and $B$ respectively.
\end{remark}

\begin{remark}
Proposition \ref{proposition5} provides a very rough lower bound for rank$_{\mathbb{Z}_2}(M_L)$. Actually, the difference between rank$_{\mathbb{Z}_2}(M_L)$ and $r-n-1$ could be arbitrarily large. Consider a link diagram $L$ on $T^2$ such that the projection of $L$ consists of $n-1$ parallel meridians and one longitude. Then one computes
\begin{center}
rank$_{\mathbb{Z}_2}(M_L)-(r-n-1)=n-2-(n-1-n-1)=n$.
\end{center}
It is an interesting question to find a precise formula for rank$_{\mathbb{Z}_2}(M_L)$.
\end{remark}

\section*{Acknowledgement}
J. Cheng, J. Xu and J. Zheng are supported by an undergraduate research project of Beijing Normal University. Z. Cheng is supported by NSFC 11771042 and NSFC 11571038.


\begin{thebibliography}{0}
\bib{Aha2012}{article}{
	author={Kazushi Ahara},
        author={Masaaki Suzuki},
	title={An integral region choice problem on knot projection},
	journal={J. Knot Theory Ramifications},
	volume={21},
	date={2012},
	number={11},
	pages={1250119, 20 pp.}}

\bib{Aid1992}{article}{
	author={H. Aida},
	title={Unknotting operation for Polygonal type},
	journal={Tokyo J. Math.},
	volume={15},
	date={1992},
	number={1},
	pages={111--121}}

\bib{Che2012}{article}{
	author={Zhiyun Cheng},
        author={Hongzhu Gao},
	title={On region crossing change and incidence matrix},
	journal={Science China Mathematics},
	volume={55},
	date={2012},
	number={7},
	pages={1487--1495}}

\bib{CH2013}{article}{
	author={Zhiyun Cheng},
        author={Hongzhu Gao},
	title={Mutation on knots and Whitney’s 2-isomorphism theorem},
	journal={Acta Mathematica Sinica, English Series},
	volume={29},
	date={2013},
	number={6},
	pages={1219--1230}}

\bib{Che2013}{article}{
	author={Zhiyun Cheng},
	title={When is region crossing change an unknotting operation?},
	journal={Math. Proc. Cambridge Philos. Soc},
	volume={155},
	date={2013},
	number={2},
	pages={257--269}}

\bib{Das2018}{article}{
	author={Oliver Dasbach},
        author={Heather M. Russell},
	title={Equivalence of edge bicolored graphs on surfaces},
	journal={Electron. J. Combin.},
	volume={25},
	date={2018},
	number={1},
	pages={Paper 1.59, 15 pp.}}

\bib{God2001}{book}{
	author={Chris Godsil},
        author={Gordon Royle},
	title={Algebraic graph theory},
	series={Graduate Texts in Mathematics},
	volume={207},
	publisher={Springer-Verlag},
        address={New York},
	date={2001},
	pages={439}}

\bib{Hay2015}{article}{
	author={Kenta Hayano},
        author={Ayaka Shimizu},
        author={Reiko Shinjo},
	title={Region crossing change on spatial-graph diagrams},
	journal={J. Knot Theory Ramifications},
	volume={24},
	date={2015},
	number={8},
	pages={1550045, 12 pp.}}

\bib{Ino2016}{article}{
	author={Ayumu Inoue},
        author={Ryo Shimizu},
	title={A subspecies of region crossing change, region freeze crossing change},
	journal={J. Knot Theory Ramifications},
	volume={25},
	date={2016},
	number={14},
	pages={1650075, 9 pp.}}

\bib{Mur1985}{article}{
	author={H. Murakami},
	title={Some metrics on classical knots},
	journal={Math. Ann.},
	volume={270},
	date={1985},
	pages={35--45}}

\bib{Mur1989}{article}{
	author={H. Murakami},
        author={Y. Nakanishi},
	title={On a certain move generating link-homology},
	journal={Math. Ann.},
	volume={284},
	date={1989},
	pages={75--89}}

\bib{Sch1998}{article}{
	author={Martin Scharlemann},
	title={Crossing changes},
	journal={Chaos Solitons Fractals},
	volume={9},
	date={1998},
	number={4-5},
	pages={693--704}}

\bib{Shi2014}{article}{
	author={Ayaka Shimizu},
	title={Region crossing change is an unknotting operation},
	journal={J. Math. Soc. Japan},
	volume={66},
	date={2014},
	number={3},
	pages={693--708}}
	
\bib{Sil2015}{article}{
	author={Daniel Silver},
	author={Susan Williams},
	title={On the component number of links from plane graphs},
	journal={J. Knot Theory Ramifications},
	volume={24},
	date={2015},
	number={1},
	pages={1520002, 5 pp.}}
	
\bib{Whi1933}{article}{
	author={Hassler Whitney},
	title={2-isomorphic graphs},
	journal={Amer. J. Math.},
	volume={55},
	date={1933},
	number={},
	pages={245--254}}
\end{thebibliography}
\end{document}